\documentclass[reqno,11pt]{amsart}
\usepackage{a4wide,color,eucal,enumerate,mathrsfs}
\usepackage[normalem]{ulem}
\usepackage{amsmath,amssymb,epsfig,amsthm} 
\usepackage[latin1]{inputenc}
\usepackage{psfrag}
\usepackage{hyperref,cleveref}

\def\loc{{\mathop\mathrm{\,loc\,}}}

\def\ez{\epsilon}
\def\eps{\varepsilon}

\def\bint{{\ifinner\rlap{\bf\kern.35em--}
\int\else\rlap{\bf\kern.45em--}\int\fi}\ignorespaces}

\def\bbint{{\ifinner\rlap{\bf\kern.35em--}
\hspace{0.078cm}\int\else\rlap{\bf\kern.45em--}\int\fi}\ignorespaces}

\def\diam{{\mathop\mathrm{\,diam\,}}}

\def\epsilon{\varepsilon}

\newtheorem{thm}{Theorem}[section]
\newtheorem{prop}[thm]{Proposition}
\newtheorem{conj}[thm]{Conjecture}
\numberwithin{equation}{section}

\theoremstyle{remark}

\def\bint{{\ifinner\rlap{\bf\kern.35em--}
\int\else\rlap{\bf\kern.45em--}\int\fi}\ignorespaces}

\usepackage{graphicx}
\usepackage{import}
\usepackage{xifthen}
\usepackage{pdfpages}
\usepackage{transparent}

\title[Bounded weak quasiregular mappings that fail to be quasiregular]{Bounded Continuous weak quasiregular mappings that fail to be quasiregular}
\author{Stanislav Hencl and Yi Ru-Ya Zhang}
\date{\today}

\address{Department of Mathematical Analysis, Charles University,
	So\-ko\-lovsk\'a 83, 186~00 Prague 8, Czech Republic}
\email{hencl@karlin.mff.cuni.cz}

\address{State Key Laboratory of Mathematical Sciences, Academy of Mathematics and Systems Science, Chinese Academy of Sciences, Beijing 100190, China}
\address{Academy of Mathematics and Systems Science, the Chinese Academy of Sciences, Beijing 100190, China}
\email{yzhang@amss.ac.cn}

 \thanks{ The first author was supported by the grant GA\v{C}R P201/24-10505S. The last author is funded by the National Key R\&D Program of China (Grant No. 2025YFA1018400 \&  No. 2021YFA1003100), NSFC grant No. 12288201 \& No. 12571128, the Chinese Academy of Sciences, and CAS Project for Young Scientists in Basic Research, Grant No. YSBR-031. }

\subjclass[2020]{30C65, 46E35 }
\keywords{Weakly quasiregular mappings, quasiregular mappings, distributional Jacobian.}

\begin{document}
\begin{abstract}
We show that, in dimensions $n\geq 3$, continuity and boundedness do not
restore the Sobolev regularity conjecture of Iwaniec and Martin for weakly
quasiregular mappings below the critical exponent. For every bounded domain
$\Omega\subset\mathbb R^n$ and every $1\leq p<nK/(K+1)$, we construct a
bounded continuous weakly $K$-quasiregular mapping
$$
    f\in W^{1,\,p}(\Omega;\,\mathbb R^n)\cap C(\Omega;\,\mathbb R^n)
   \cap L^\infty(\Omega;\mathbb R^n)
$$
which fails to be quasiregular. We further construct weakly quasiregular
mappings whose singular sets have Hausdorff dimension arbitrarily close to the
maximal size permitted by their Sobolev regularity. These examples show 
that, the almost-everywhere sign condition on the Jacobian is too weak to serve as an orientation-preserving hypothesis below $W^{1,n}$. In contrast, we show that, for $n-1<p<n$, quasiregularity follows once this condition is replaced by a one-sided condition on the distributional degree (together with boundedness).
\end{abstract}
 
\maketitle
\section{Introduction}
Let $\Omega\subset \mathbb R^n$ be an open set and $f\colon\Omega\to \mathbb R^n$ be a mapping of Sobolev class $W^{1,\,p}_{\loc}(\Omega;\,\mathbb R^n), p\ge 1$. We write $Df$ as its distributional gradient, and denote its Jacobian determinant by $J_f$. 
Then a mapping $f\in W^{1,\,p}_{\loc}(\Omega;\,\mathbb R^n)$ is said to be weakly $K$-quasiregular if $f$ satisfies
$J_f\ge 0 $ a.e. or $J_f\le 0$ a.e., and
$$\max_{\xi\in \mathbb S^{n-1}} |Df(x)\xi|\le K \min_{\xi\in \mathbb S^{n-1}} |Df(x)\xi| \ \text{ a.e.}$$
If it is further assumed that $f\in W^{1,\,n}$, then we say that $f$ is $K$-quasiregular.

A remarkable result by Iwaniec--Martin \cite[Theorem 1]{IM1993} states that, when $n=2l\ge 4, l\in\mathbb N$, a weakly $1$-quasiregular mapping $f\in W^{1,\,l}_{\loc}(\Omega;\,\mathbb R^n)$ is indeed $1$-quasiregular. Moreover, they give an example that, for each $p\in \left[1,\,\frac{nK}{K+1}\right)$, there is a weakly $K$-quasiregular mapping $f\in W^{1,\,p}_{\loc}(\Omega;\,\mathbb R^n)$ which is not quasiregular. 
Then they formulated the following conjecture \cite[Conjecture 12.9]{IM1993}: 
\begin{conj}
    Every weakly $K$-quasiregular mapping $f\in W^{1,\,p}_{\loc}(\Omega;\,\mathbb R^n)$ with $p\ge \frac{nK}{K+1}$ is $K$-quasiregular. 
\end{conj}
Meanwhile, they also  conjectured that a weaker $L^p$-version holds:
\begin{equation}\label{weak IM}
    f\in L^{nK}_{\loc}(\Omega;\,\mathbb R^n) \  \text{ is weakly $K$-quasiregular } \quad \Longrightarrow \quad  f \text{ is  $K$-quasiregular. }
\end{equation}

When $n=2$, the conjectured range by Iwaniec--Martin  for the Sobolev regularity is known to be optimal; see \cite[Theorem 14.4.3 \& Theorem 14.4.6.]{AIM2009}, and also \cite{A1994, AIS2001, PV2002} together with the reference therein for the literature. Especially,  the continuity of a weakly quasiregular mapping (or a solution to the Beltrami equation) is a decisive factor in determining its quasiregularity in \cite{PV2002}.

Later, the problem for $n\ge 3$ has been studied via quasiconvexity method in the calculus of variation \cite{Y1996, Y2001, I2002}. Moreover, 
Zhuomin Liu gave an alternative proof of Liouville theorem via a perturbation method in \cite{L2013}, and also proved the rigidity of conformal Sobolev maps under a second-order differentiability assumption \cite{L20132}. 

In this manuscript, we construct the following example to disprove \eqref{weak IM}, which is a genuine sharpening of the original example by Iwaniec-Martin.
\begin{thm}\label{main thm}
    When $n\ge3$, for any bounded domain $\Omega\subset \mathbb R^n$ and every $p\in \left(1,\,\frac{nK}{K+1}\right)$, there is a weakly $K$-quasiregular mapping 
    $$f\in W^{1,\,p}(\Omega;\,\mathbb R^n)\cap C(\Omega;\,\mathbb R^n)
   \cap L^\infty(\Omega;\mathbb R^n)$$
    which is not quasiregular.
\end{thm}

By a similar construction as in Theorem~\ref{main thm}, we obtain the following.
\begin{thm}\label{dimension}
Let $n\geq 3$, $K\geq 1$, and let $\epsilon>0$. 
Then there exists a weakly $K$-quasiregular mapping
$$
    f\in W^{1,1}_{\operatorname{loc}}(\Omega;\mathbb R^n)
$$
whose singular set
$$    S_f    :=    \left\{        x\in\Omega:        \limsup_{r\to 0^+}        \bint_{B(x,r)} |f|\,dx        =\infty    \right\}
$$
has Hausdorff dimenson at least 
$$    \min\left\{\frac{nK}{K+1},\,n-1\right\} -\epsilon.$$
    
    In particular, for any $p\in \left(1,\,n\right)$, one can find a weakly $K_p$-quasiregular mapping $f\in W^{1,\,p}_{\loc}(\Omega;\,\mathbb R^n)$ with $K_p=K(n,\,p)$ so that the Hausdorff dimension of the singular set of $f$ is at least $n-p-\eps$.
\end{thm}

This result tells that, below the critical exponent, weak quasiregularity does not reduce the worst-case Sobolev singular-set size. Moreover, this is naturally related to the conjectured optimal Hausdorff dimension $\frac n{K+1}$ of the removable/nonremovable singular set for (bounded) $K$-quasiregular mappings, which was studied in \cite{IM1993} as well; see also \cite{GMP2017}.

By Theorem~\ref{main thm}, to study the conjecture of Iwaniec--Martin, it is not very restrictive to assume that $f$ is bounded or even continuous when $p\in \left[1,\,\frac{nK}{K+1}\right)$. Now for $n-1<p<n$, define the distributional Jacobian of $f\in W^{1,\,p}_{\loc}(\Omega;\,\mathbb R^n)\cap L^\infty(\Omega;\,\mathbb R^n)$ by
$$\langle\mathcal J_f,\,\varphi\rangle=-\frac 1 n\int_{\Omega} ({\rm adj}Df \cdot f) \cdot D\varphi\, dx,$$
and  for any $D\subset\subset \Omega$  Lipschitz, the distributional degree of $f$ on $D$ by
$$\int_{\mathbb R^n} {\rm Deg}(f,\,D,\,y) {\rm div}(g(y))\,dy=\int_{\partial D} ({\rm adj}Df \cdot g(f(x))\cdot \nu_{D}\,d\mathscr H^{n-1}(x), \quad   \forall g\in C^1(\Omega),
$$
where ${\rm adj} Df$ satisfies
$$({\rm adj} Df) Df= J_f \,Id;$$
see e.g. \cite{D2012} for more details. Note that for continuous $f$ the above notion of distributional degree actually coincides with the usual topological degree (see \cite{D2012}). 
In particular, according to \cite[Theorem 1.1 \& Proposition 5.4]{D2012} (see also \cite{M1990, M1993, DG2010}), we have the following proposition.
\begin{prop}\label{positive part}
Let $n-1<p<n$, and let
$$
    f\in W^{1,p}_{\operatorname{loc}}(\Omega;\mathbb R^n)
        \cap L^\infty_{\operatorname{loc}}(\Omega;\mathbb R^n).
$$
Assume that $f$ weakly preserves orientation in the sense that, for every
$x\in\Omega$ and for almost every
$0<r<\operatorname{dist}(x,\partial\Omega)$,
$$
    \operatorname{Deg}(f,B(x,r),y)\ge 0
    \quad\text{for a.e. }y\in\mathbb R^n .
$$
Assume moreover that
$$
    |Df(x)|^n \le K\, |\det Df(x)|
    \quad\text{for a.e. }x\in\Omega .
$$
Then
$$
    f\in W^{1,n}_{\operatorname{loc}}(\Omega;\mathbb R^n).
$$
In particular, $f$ is quasiregular.
The same conclusion holds with nonpositive degree, after composing $f$
with an orientation-reversing isometry of $\mathbb R^n$.
\end{prop} 

\begin{proof}
Let $\mathcal J_f$ denote the distributional Jacobian of $f$. 
By the nonnegative degree assumption, $f\in WOP(\Omega)$. Hence, by
De Philippis' result \cite[Proposition 5.4]{D2012}, $\mathcal J_f$ is a positive Radon measure. Moreover,
since $p>n-1$, the absolutely continuous part of $\mathcal J_f$ is
$$
    (\mathcal J_f)^{ac} = \det Df\,\mathcal L^n .
$$
Consequently $\det Df\ge0$ a.e. Let $U\Subset\Omega$. Then
$$
    \int_U |Df|^n\,dx
    \le K\int_U \det Df\,dx
    \le K\,\mathcal J_f(U)<\infty .
$$
Thus $f\in W^{1,n}_{\loc}(\Omega;\mathbb R^n)$, and the
distortion inequality gives quasiregularity.
\end{proof}

Note that, in Proposition~\ref{positive part}, we replaced non-signchanging  almost everywhere of $J_f$ to that of the distributional degree. Indeed it was noticed by \cite{BHM2017} that nonnegative degree is a more appropriate orientation-preserving notion for mappings in Sobolev class. 
Especially, when $K$ is sufficiently large so that
$$\frac{nK}{K+1}>n-1,$$
Proposition~\ref{positive part} tells that the non-signchanging of the distributional degree is essential for weakly $K$quasiregular mappings to be $K$-quasiregular.  
Namely, in the range $p>n-1$, replacing this pointwise sign condition
by the one-sided degree condition $f\in WOP$ restores the implication
from weak quasiregularity to quasiregularity.

\section{Proof of Theorem~\ref{main thm} and Theorem~\ref{dimension}}

Before giving the proof of Theorem~\ref{main thm}, let us first briefly recall the construction by Iwaniec--Martin \cite[Proof of Theorem 12.1]{IM1993}. They started with an \emph{exact packing of $\Omega$ by balls}, i.e. for every open set $\Omega\subset \mathbb R^n$ and any $\delta_0>0$, there exists a disjoint family 
$$\mathcal F=\{B_j\colon 1\le j<\infty\}$$
of open balls $B_j\subset \Omega$ so that $\diam (B_j)<\delta_0$ and 
$$\left|\Omega\setminus \bigcup_{j=1}^\infty B_j\right|=0;$$
see \cite[Section 1.5.1, Corollary 2]{EG1992}. Then for each ball $B_j=:B({\mathbf y}_j,\,r_j)$, they define 
$$\Phi_{B_j}(\mathbf x)={\mathbf y}_j+ (\mathbf x-{\mathbf y}_j)\left(\frac{r_j}{|\mathbf x-{\mathbf y}_j|}\right)^{1+\frac 1 K}.$$
At last, they set
$$F(\mathbf {x})=\Phi_{B_j}(\mathbf {x}) \quad  \text{ when } \ \mathbf {x}\in B_j, \quad F(\mathbf {x})=\mathbf {x} \quad \text{ otherwise},$$
and proved that $F\in  W^{1,\,p}_{\loc}(\Omega;\,\mathbb R^n)$ with $1\le p<\frac{nK}{K+1}.$

\medskip

The construction of our example follows the idea of Iwaniec--Martin and is partially motivated by \cite[Section 4]{KKM2001}, but via a more delicate truncation and iteration argument. 

\begin{proof}[Proof of Theorem~\ref{main thm}]
Given a ball $B(y,r)$, $\mathbf z\in \mathbb R^n$ and $t>0$ we define a mapping
$$\Phi_{B}^{\mathbf z,t}(\mathbf x)=\mathbf z+ t(\mathbf x-\mathbf y)\left(\frac{r}{|\mathbf x-\mathbf y|}\right)^{1+\frac 1 K}.$$
\noindent{\bf Step 1: Construction of  a weakly quasiregular mapping $F$.} Fix $0<a<1$ and consider an exact packing $\mathcal F_1=\{B_{1,\,j}\}_{j=1}^\infty, B_{1,\,j}=B(\mathbf {y}_{1,\,j},\,r_{1,\,j})$ of $\Omega=:U_1$ by balls  with radius no more than $\delta_1>0$. We define 
$$
F_1(\mathbf {x}) = \left\{ \begin{array}{ll}
\Phi^{\mathbf{y}_{1,j},1}_{B_{1,\,j}}(\mathbf {x}) & \textrm{ if } \ \mathbf x\in B_{1,\,j}\setminus a B_{1,\,j}\\
\mathbf {y}_{1,\,j}+a^{-\frac 1 K-1}(\mathbf x-\mathbf {y}_{1,\,j}) & \textrm{ if } \ \mathbf  x\in aB_{1,\,j}\\
\mathbf {x}  & \textrm{ otherwise }
\end{array} \right. .$$
This give us the first continuous map $F_1$.

We next construct $F_2$. Towards this, we consider  an exact packing $\mathcal F_2=\{B_{2,\,j}\}_{j=1}^\infty$ of 
$$U_2=\bigcup_{j=1}^\infty aB_{1,\,j}$$
by balls $B_{2,\,j}=B(\mathbf {y}_{2,\,j},\,r_{2,\,j})$ with radius no more than $\delta_2>0$.
Then we define
$$
F_2(\mathbf {x}) = \left\{ \begin{array}{ll}
\Phi^{F_1(\mathbf y_{2,j}),a^{-\frac{1}{K}-1}}_{B_{2,\,j}}(\mathbf {x}) & \textrm{ if } \ \mathbf x\in B_{2,\,j}\setminus a B_{2,\,j}\\
F_1(\mathbf {y}_{2,\,j})+a^{-\frac 2 K-2}(\mathbf x-\mathbf {y}_{2,\,j}) & \textrm{ if } \  \mathbf x\in aB_{2,\,j}\\
F_1(\mathbf x) & \textrm{ otherwise }
\end{array} \right. .$$
Note that for $\mathbf{x}\in \partial a B_{2,\,j}$ we have
$$
\Phi^{F_1(\mathbf{y}_{2,j}),a^{-\frac{1}{K}-1}}_{B_{2,\,j}}(\mathbf {x})=
F_1(\mathbf{y}_{2,j})+a^{-\frac{1}{K}-1} (\mathbf{x}-\mathbf{y}_{2,j})
\left(\frac{r_{2,\,j}}{ar_{2,\,j}}\right)^{1+\frac 1 K}=F_1(\mathbf {y}_{2,\,j})+a^{-\frac 2 K-2}(\mathbf x-\mathbf {y}_{2,\,j})
$$
and for $\mathbf{x}\in \partial B_{2,\,j}$, since $F_1$ is linear on $U_2$ we have
$$
\Phi^{F_1(\mathbf{y}_{2,j}),a^{-\frac{1}{K}-1}}_{B_{2,\,j}}(\mathbf {x})=
F_1(\mathbf{y}_{2,j})+a^{-\frac{1}{K}-1} (\mathbf{x}-\mathbf{y}_{2,j})
\left(\frac{r_{2,\,j}}{r_{2,\,j}}\right)^{1+\frac 1 K}= F_1(\mathbf{x})
$$
so the mapping is continuous.

Now we iterate this construction. Namely, given $\mathcal F_{k-1}$,   set 
$$U_k=\bigcup_{j=1}^\infty aB_{k-1,\,j},$$
and consider an exact packing $\mathcal F_k=\{B_{k,\,j}\}_{j=1}^\infty$ of
$U_k$ by balls $B_{k,\,j}=B(\mathbf {y}_{k,\,j},\,r_{k,\,j})$ with radius no more than $\delta_k>0$. Then define inductively
$$
F_k(\mathbf {x}) = \left\{ \begin{array}{ll}
\Phi^{F_{k-1}(\mathbf{y}_{k,j}),a^{(k-1)\left(-\frac{1}{K}-1\right)}}_{B_{k,\,j}}(\mathbf {x}) & \textrm{ if } \ \mathbf x\in B_{k,\,j}\setminus a B_{k,\,j}\\
F_{k-1}(\mathbf{y}_{k,j})+a^{-\frac {k} K-k}(\mathbf x-\mathbf {y}_{k,\,j})   & \textrm{ if } \  \mathbf x\in aB_{k,\,j}\\
F_{k-1}(\mathbf x) & \textrm{ otherwise }
\end{array} \right. $$
and again it is not difficult to check that this mapping is continuous. 
We further require that $\delta_k$ are chosen to satisfy
\begin{equation}\label{finite sum}
    \sum_{k=1}^\infty a^{-\frac k K-k} \delta_k<\infty. 
\end{equation}
This gives us a sequence of continuous mappings  $\{F_k\}$.

Moreover,  we claim that $F_k$ are converging uniformly. Indeed, according to the definition of $F_k$, the image of 
\begin{equation}\label{image}
F_k(B_{k,\,j}\setminus a B_{k,\,j})\subset  B(F_{k-1}(\mathbf{y}_{k,j}),\,a^{-\frac k K-k}\delta_k),
\end{equation}
and by \eqref{finite sum} it it easy to see that $F_k$ are uniformly bounded. Using the  Cauchy convergence \eqref{finite sum} it is also standard to show that $F_k$ converge uniformly to a continuous function $F$. 


In addition, according to our construction, it follows that
\begin{equation}\label{zero null set}
    \left|\Omega\setminus \bigcup_{k=1}^\infty\bigcup_{j=1}^\infty (B_{k,\,j}\setminus aB_{k,\,j})\right|=0,
\end{equation} 
and for every $\mathbf x\in B_{k,\,j}\setminus aB_{k,\,j}$,
\begin{equation}\label{easy}
F_{l}(\mathbf x) = F_k(\mathbf x)= F_{k-1}(\mathbf{y}_{k,\,j})+  a^{-(k-1) (1+\frac 1 K)}(\mathbf x-\mathbf y_{k,\,j})\left(\frac{r_{k,\,j}}{|\mathbf x-\mathbf y_{k,\,j}|}\right)^{1+\frac 1 K}, \quad \forall l\ge k,
\end{equation}
which coincides with $\Phi_{B_{k,\,j}}$ up to a translation and a dilation. 
Since each $\Phi_{B_{k,\,j}}$ is $K$-quasiregular and $\det DF\le 0$ almost everywhere, it follows that $F$ is weakly $K$-quasiregular.  However, $F$ is not $K$-quasiregular as its degree is positive on balls $B_{k,j}$ as our $F$ is just a positive linear transformation (i.e. $F(x)=a_{k,j}+a^{-(k-1)(\frac{1}{K}+1)} x$, see \eqref{easy}) on $\partial B_{k,j}$. 

\medskip

\noindent{\bf Step 2: Energy estimate of $F$.} Now we compute $|DF|$. Observe that when $\mathbf x\in B_{k,\,j}\setminus a B_{k,\,j}$
$$D\Phi_{B_{k,\,j}}(\mathbf x) =a^{-\frac {K+1} K(k-1)}\left(\frac{r_{k,\,j}}{|\mathbf {x}-\mathbf {y}_{k,\,j}|}\right)^{1+\frac 1 K} \left[Id-\frac{K+1}{K} \frac{(\mathbf {x}-\mathbf {y}_{k,\,j})\otimes (\mathbf {x}-\mathbf {y}_{k,\,j})}{|\mathbf {x}-\mathbf {y}_{k,\,j}|^2}\right].$$
Thus
$$|D\Phi_{B_{k,\,j}}(\mathbf x)|\le C(K)a^{-\frac {K+1} K(k-1)} r_{k,\,j}^{1+\frac 1 K}|\mathbf {x}-\mathbf {y}_{k,\,j}|^{-\left(1+\frac 1 K\right)},$$
and then
\begin{multline}\label{energy in one}
\int_{B_{k,\,j}\setminus a B_{k,\,j}} |D\Phi_{B_{k,\,j}}|^p\,dx\le  C(n,\,K) a^{-\frac {K+1} K(k-1)p} r_{k,\,j}^{p\left(1+\frac 1 K\right)}\int_{ar_{k,\,j}}^{r_{k,\,j}} s^{-p\left(1+\frac 1 K\right)+n-1}\, ds \\
\le C(n,\,K,\,a,\,p) a^{-\frac {K+1} K kp} r_{k,\,j}^{p\left(1+\frac 1 K\right)} r_{k,\,j}^{-p\left(1+\frac 1 K\right)+n}(1-a^{-p\left(1+\frac 1 K\right)+n})\le C(n,\,K,\,a,\,p) a^{-\frac {K+1} K kp}r_{k,\,j}^n,
\end{multline}
where we applied
$$-p\left(1+\frac 1 K\right)+n-1> -1 \quad \Longleftrightarrow \quad p<\frac{nK}{K+1}.$$
Now by summing over all $B_{k,\,j}$ and applying \eqref{energy in one} together with \eqref{zero null set}, we conclude that, for each $k_0\in \mathbb N$,
\begin{align*}
\int_{\Omega} |DF_{k_0}|^p\,dx\le & \ \sum_{k=1}^{k_0}\sum_{j=1}^\infty\int_{B_{k,\,j}\setminus a B_{k,\,j}} |D\Phi_{B_{k,\,j}}|^p\,dx + \sum_{j=1}^{\infty}\int_{B_{k_0,\,j}} |D\Phi_{B_{k_0,\,j}}|^p\,dx \\
\le  & \ C(n,\,K,\,a,\,p)\sum_{k,\,j=1}^\infty a^{-kp\frac{K+1} K }r_{k,\,j}^n  \\ 
= & \ C(n,\,K,\,a,\,p)\sum_{k=1}^\infty a^{-kp\frac{K+1} K }|U_{k}|\\
=&\ C(n,\,K,\,a,\,p)\sum_{k=1}^\infty a^{-kp\frac{K+1} K } a^{nk}|U_1|
\le C(n,\,K,\,a,\,p) |\Omega|<\infty, 
\end{align*}
where we applied that $\mathcal F_k$ is an exact packing of $U_k$ and $\frac{K}{K+1}n>p$.
Thus, we conclude $F\in W^{1,\,p}(\Omega)$ by the weak convergence of $F_k\rightharpoonup F$ in $W^{1,\,p}(\Omega)$, up to further passing to a subsequence. 
\end{proof}

\begin{proof}[Proof of Theorem~\ref{dimension}]
The proof of Theorem~\ref{dimension} is similar to the one of Theorem~\ref{main thm}, and we sketch the proof here. 

Let $a\in (0,\,1)$ be small enough to be determined later. our construction is similar to the one in the proof of Theorem~\ref{main thm}, except
\begin{itemize}
    \item when $p\geq n/2$, we use $\Phi_{B}^{\mathbf z,t}(\mathbf x)$ in the construction; 
\item In the case $p<n/2$, 
we replace $\Phi_{B}^{\mathbf z,t}(\mathbf x)$ by 
    $$\Psi^{\mathbf z,t}_B(\mathbf x)=\mathbf z+ t(\mathbf x-\mathbf y)\left(\frac{r}{|\mathbf x-\mathbf y|}\right)^{1+ K},$$
which is a $K$-quasiconformal mapping as well.
\end{itemize}

We choose $\delta>0$ small enough and we set 
\begin{equation}\label{choosedelta}
\delta_k=\begin{cases}
(1+\delta)^k a^{(1+\frac{1}{K})k} &\text{ for }\frac{n}{2}\leq p< n ,\\
(1+\delta)^k a^{(1+K)k} &\text{ for }1\leq p< \frac{n}{2}.\\
\end{cases}
\end{equation}
Note that in contrary to \eqref{finite sum} we now have infinite sum which will mean that images of $B_{k,j}$ are very often quite big in view of \eqref{image}. 

Now, when we choose the exact packing $\mathcal F_k=\{B_{k,\,j}\}$ of $U_k$ we choose many balls with radius exactly $\delta_k$ and the rest is covered by potentially smaller balls. We require that, for some $C_0=C_0(n)$, there are at least $C_0\left(\frac{a\delta_{k-1}}{\delta_k}\right)^n$-many balls $\hat B_{k,\,j}\in\mathcal F_k$ of radius $\delta_k$ inside the ball of radius $\delta_{k-1}$ from $\mathcal F_{k-1}$. This is possible as we have balls of radius $a\delta_{k-1}$ and put in them much smaller balls of radius $\delta_k$ (recall that $0<a<1$ is really small) and we can cover at least $C_0 \%$ of the volume of the ball of radius $a\delta_{k-1}$ where the constant depends only on the packing properties of $\mathbb R^n$. In this way we obtain a Cantor type set (with bad behavior of $F$) 
as the nested intersection of at least 
$$
\Bigl(C_0\left(\frac{a\delta_{k-1}}{\delta_k}\right)^n\Bigr)^{k-1}\text{ ball from }B_{k,j}
\text{ of radius exactly }\delta_k. 
$$
With the help of \eqref{choosedelta} we obtain that the Hausdorff dimension $d$ of this set can be computed from
$$
\Bigl(C_0\left(a^{-\frac{1}{K}}\right)^n\Bigr)^{k-1}
\bigl((1+\delta)^k a^{(1+\frac{1}{K})k}\bigr)^d\approx  1\text{ for }p\geq \frac{n}{2}
$$
and from 
$$
\Bigl(C_0\left(a^{-K}\right)^n\Bigr)^{k-1}
\bigl((1+\delta)^k a^{(1+K)k}\bigr)^d\approx 1\text{ for }p< \frac{n}{2}. 
$$
Now $C_0(n)$ and $\delta$ are fixed so we can choose $a$ small enough so that $C_0$ and $\delta$ are not important (up to something like $a^{\epsilon}$) in the computation above and we get that the dimension of the Cantor type set is at least 
$$
d=n\frac{1}{K+1}-\frac{\epsilon}{2} \  \text{ for } \ \frac{n}{2}\leq p< n \quad \text{ or }\quad 
d=n\frac{K}{K+1}-\frac{\epsilon}{2}\  \text{ for } \ 1\leq p<\frac{n}{2}. 
$$


In addition, a direct calculation similar to Step 2 of Theorem~\ref{main thm} yields 
$$\int_{\Omega} |DF_k|^p\,dx\le C(n,\,K,\,a,\,p)|\Omega|<\infty$$
whenever $1< p<\frac{n}{K+1}$ in the case $1< p<\frac{n}{2}$ and as before whenever $p<\frac{nK}{K+1}$ in the case $\frac{n}{2}\leq p<n$. Therefore, up to relabeling the sequence, there exists $F\in W^{1,\,p}(\Omega)$ so that $F_k\to F$ weakly in $W^{1,\,p}(\Omega)$ and strongly in $L^p$. 

Now for our fixed $\ez>0$ and every $p\in\left(1,\,n\right)$, we select the parameter $K_p$ and generate $F$ based on the range of $p$ as follows:
\begin{itemize}
    \item When $p\ge \frac n 2$, choose $K_p>\frac{p}{n-p}$ sufficiently close to $\frac{p}{n-p}$.
    \item When $p<\frac n 2$, choose $K_p<\frac n p -1$ sufficiently close to $\frac n p -1$.
\end{itemize}
We get the dimension of the Cantor type set as 
$$
d=
\begin{cases}
n\frac{1}{K+1}-\frac{\epsilon}{2}=n\frac{1}{\frac{p}{n-p}+1}-\epsilon=n-p-\epsilon&\text{ for }\frac{n}{2}\leq p<n,\\
d=n\frac{K}{K+1}-\frac{\epsilon}{2}=n\frac{\frac{n}{p}-1}{\frac{n}{p}-1+1}-\epsilon=n-p-\epsilon&\text{ for }1\leq p<\frac{n}{2}.\\
\end{cases}
$$

Let $p\geq \frac{n}{2}$ and let $x$ be a point of our Cantor type set. For each $k$ we  can find $B_{k,j}$ so that that $x\in B_{k,j}$ 
and analogously to \eqref{image} using \eqref{choosedelta} we obtain that 
$$
F_k(B_{k,j}\setminus a B_{k,j})=
B(F_{k-1}(y_{k,j}),(1+\delta)^{k})\setminus B(F_{k-1}(y_{k,j}),a(1+\delta)^{k-1})
$$
and from our construction it follows that $F=F_k$ on $B_{k,j}\setminus a B_{k,j}$. 
It follows that for each $r$ we can find $k$ so that $x\in B_{k,j}\subset B(x,r)$ but $x\in B_{k-1,\tilde{j}}\not\subset B(x,r)$ (and thus $|B(x,r)|\leq C(a)|B_{k,j}|$ as radius of $B_{k,j}$ is $\delta_k$ and radius of $B_{k-1,\tilde{j}}$ is $\delta_{k-1}$) and then
 we have
$$
-\hskip -12pt\int_{B(x,r)} |F|\geq C(a)(1+\delta)^{k-1}
$$
which tends to $\infty$ as $k\to \infty$. 
The argument for $x$ that belongs to the singular set for $p<\frac{n}{2}$ is similar. 
\end{proof}

\end{document}